\documentclass[reqno,a4paper,12pt]{article}
\usepackage{hyperref}
\usepackage{amsmath,amssymb,amsthm,authblk}
\usepackage[pdftex]{graphicx}
\title{\bf\large draft 0}
\date{\vspace{-5ex}}
\author[1]{\small Adel M. Al-Mahdi}\author[2] {\small Mohammad M. Al-Gharabli}\author[3] {\small Mohammad
Kafini} \author[4] {\small Shadi Al-Omari} \affil[1,2,4]{\small\it
The Preparatory Year Program, King Fahd University of Petroleum and
Minerals, Dhahran 31261, Saudi Arabia.} \affil[3]{\small\it
Department of Mathematics and Statistics, King Fahd University of
Petroleum and Minerals, Dhahran 31261, Saudi Arabia.}
\affil[1,2,3,4]{The interdisciplinary research center in
construction and building materials, King Fahd University of
Petroleum and Minerals, Dhahran 31261, Saudi Arabia.}

\begin{document}
\title{\bf\large Stability  result for a viscoelastic wave equation in the presence of
finite and infinite memories} \maketitle

\begin{abstract}In this paper, we are concerned with the following viscoelastic wave
equation
\begin{equation*} \label{1}
u_{tt}-\nabla u +\int_0^t g_1 (t-s)~ div(a_1(x) \nabla u(s))~ ds +
\int_0^{+ \infty} g_2 (s)~ div(a_2(x) \nabla u(t-s)) ~ds = 0,
\end{equation*}
in a bounded domain $\Omega$. Under suitable conditions on $a_1$ and
$a_2$ and for a wide class of relaxation functions $g_1$ and $g_2$.
We establish a general decay result. The proof is based on the
multiplier method and makes use of convex functions and some
inequalities. More specifically, we remove the constraint imposed on
the boundedness condition on the initial data $\nabla u_{0}$. This
study generalizes and improves previous literature outcomes.
\end{abstract}\bf{Keywords:}\rm \;   Stability,  Finite and Infinite Memories,  Relaxation Functions, Multiplier Method, Convexity.
\newtheorem{theorem}{Theorem}[section]
\newtheorem{lemma}[theorem]{Lemma}
\newtheorem{remark}[theorem]{Remark}
\newtheorem{definition}[theorem]{Definition}
\newtheorem{proposition}[theorem]{Proposition}
\newtheorem{example}{Example}
\newtheorem{corollary}[theorem]{Corollary}
\allowdisplaybreaks
\section{Introduction} Let us consider an $n$-dimensional body occupies with a bounded  open set $\Omega \subseteq \mathbb{R}^n (n \geq 1)$
with smooth boundary $\partial \Omega$.  Let $x \rightarrow u(x,t)$
be the position of the material particle $x$ at time $t$. So that,
$\forall x \in \Omega, \forall t>0$,  the corresponding motion
equation is
\begin{equation} \label{main}
	\begin{array}{ll}
		u_{tt}-\nabla u +\int_0^t g_1 (t-s) div(a_1(x) \nabla u(s)) ds + \int_0^{+ \infty} g_2 (s) div(a_2(x) \nabla u(t-s))ds = 0, \\
		u(x, t) = 0,\\
		u(x, -t)=u_0 (x, t), u_t (x, 0)=u_1 (x).
	\end{array}
\end{equation} The
functions $g_1$ and $g_2$ are called the relaxation (kernels) functions and they are positive non-increasing and
defined on $\mathbb{R}^{+}$. The functions  $a_1$ and $a_2$ are
essentially bounded non-negative defined on $\Omega$. Here,  $u_0$
and $u_1$ are the given initial data. This model of materials
consisting of an elastic part (without memory) and a viscoelastic,
part, where the dissipation given by the memory is effective.\\
\textbf{In this paper}, we are concerned with the above viscoelastic
wave problem \eqref{main} and  mainly interested in the asymptotic
behavior of the solution $u$ when $t$ tends to infinity. In fact, We
prove that the solutions of the corresponding viscoelastic model
decay to zero and no matter how small is the viscoelastic part of
the material. Note that the above model is dissipative, and the
dissipation is given by the memory term only and  the memory is
effective only in a part of the body. For materials with memory the
stress depends not only on the present values but also on the entire
temporal history of the motion. Therefore, we have also to prescribe
the history of $u$ before 0. Here, we assume that $u(x, -t)=u_0 (x,
t),   \forall x \in \Omega, \forall
t>0.$ Let us mention some other papers related to the problems we
address. We start our literature review with the pioneer work of Dafermos
\cite{Dafermos}, in 1970, where the author discussed a certain
one-dimensional viscoelastic problem, established some existence
results, and then proved that, for smooth monotone decreasing
relaxation functions, the solutions go to zero as $t$ goes to
infinity. However, no rate of decay has been specified. In Dafermos
\cite{dafermos1970abstract}, a similar result, under a ity
condition on the kernel, has been established. After that a great
deal of attention has been devoted to the study of viscoelastic
problems and many existence and long-time behavior results have been
established. Hrusa \cite{hrusa1985global} considered a one-dimensional
nonlinear viscoelastic equation of the form
\begin{equation}
	u_{tt} - cu_{xx}  +\int_0^t m(t-s)  \left ( \psi (u_x(x, s))
	\right)_x~ ds = f(x, t)
\end{equation}
and proved several global existence results for large data. He also
proved an exponential decay for strong solutions when $m(s)=e^{-s}$
and $\psi$ satisfies certain conditions. Dassios and Zafiropoulos
\cite{Dassios} studied a viscoelastic problem in $\mathbb{R}^3$ and
proved a polynomial decay results for exponentially decaying
kernels. After that, a very important contribution by Rivera was
introduced. In 1994, Rivera \cite{Rivera} considered
equations for linear isotropic homogeneous viscoelastic solids of
integral type which occupy a bounded domain or the whole space
$\mathbb{R}^n$. in the bounded domain case, and for exponentially
decaying memory kernel and regular solutions, he showed that the sum
of the first and the second energy decays exponentially. For the
whole-space case and for exponentially decaying memory kernels, he
showed that the rate of decay of energy is of algebraic type and
depends on the regularity of the solution. this result was later
generalized to a situation, where the kernel is decaying
algebraically but not exponentially by Cabanillas and Rivera
\cite{Cabanillas}. In the paper, the authors considered the case of
bounded domains as well as the case when the material is occupying
the entire space and showed that the decay of solutions is also
algebraic, at a rate which can be determined by the rate of decay of
the relaxation function. This latter result was later improved by
Baretto et al. \cite{rivera1996decay}, where equations related to linear
viscoelastic plates were treated. Precisely, they showed that the
solution energy decays at the same rate of the relaxation function.
For partially viscoelastic material, Rivera and Salvatierra
\cite{Riv-Peres} showed that the energy decays
exponentially, provide the relaxation function decays in a similar
fashion and the dissipation is acting on a part of the domain near
to the boundary. See also, in this direction, the work of Rivera and
Oquendo \cite{Riv-Oqu}. The uniform decay of solutions for the
viscoelastic wave equation
\begin{equation}
	u_{tt} - k_0 \nabla u  +\int_0^t div  \left[ a(x) g(t- \tau ) \nabla
	u(\tau) \right] d \tau +b(x) h(u_t) + f(u) =0,
\end{equation} was investigated by Cavalcanti and Oquendo \cite{canti1} where they considered
the condition  $a(x)+b(x) \geq \delta> 0$. They established
exponential and  polynomial stability results based on some
conditions on  $g$ and the linearity of the function $h$. After
that, Guesmiaa and Messaoudib \cite{a1a2} extended the work of
\cite{canti1} and they establish a general decay result for
\eqref{main} under the same conditions on $a_1$ and $a_2$ used in
\cite{canti1} and for some other conditions for the relaxation
functions $g_1$ and $g_2$, from which the usual exponential and
polynomial decay rates are only special cases. More precisely, they
used the conditions $g_1^{\prime}(t)\le
-\xi(t)g_1(t),\hspace{0.15in}\forall t\ge 0,$ and
\begin{equation}\label{gg}
	\int_{0}^{+\infty}\frac{g_2(s)}{G^{-1}(-g_2^{\prime}(s))}ds+\sup_{s\in\mathbb{R}_+}\frac{g_2(s)}{G^{-1}(-g_2^{\prime}(s))}<+\infty,
\end{equation}
such that
\begin{equation}\label{E:3}
	G(0)=G^{\prime}(0)=0\hspace{0.05in}\text{and}\hspace{0.05in}\lim_{t\rightarrow
		+\infty}{G^{\prime}(t)}=+\infty,
\end{equation}
where $G:\mathbb{R}_+\to \mathbb{R}_+$  is an increasing strictly
convex function.\\ \textbf{In the present work}, we extend the works
of \cite{canti1} and \cite{a1a2}. In fact, we consider \eqref{main}
and under the same conditions on $a_1$ and $a_2$ used in
\cite{canti1} and for a large class of the relaxation functions, we
prove  a general decay result. More precisely, we assume that the
relaxations functions $g_1$ and $g_2$ are satisfying
$$g_i^{\prime}(t)\le -\xi(t) H_i(g_i(t)),\hspace{0.15in}\forall t\ge
0.$$ In fact, our result allows a large class of relaxation
functions and improves the decay rates in some earlier papers. The
paper is organized as follows. In Section 2, we present some
material needed for our work. Some essential lemmas are presented
and established in Section 3. Section 4 contains the statement and
the proof of our main result. We end our paper by giving some
illustrating examples in Sections 5.
\section{Preliminaries}
In this section, we present some material needed in the proof of
our main result. Through this paper, we  use $c$ to denote a positive generic constant. Now, we start with the following assumptions: \\
(A1) $g_i : \mathbb{R}^+ \longrightarrow \mathbb{R}^+ $ are
differentiable non-increasing functions such that
\begin{equation*}
	g_i (0) > 0, ~~ i=1,2, ~~~~~ 1-||a_1||_{\infty} \int_0^{+\infty}
	g_1(s)ds - ||a_2||_{\infty} \int_0^{+\infty} g_2(s) ds = l >0.
\end{equation*}and there exists  $C^1$ functions $H_i:\mathbb{R}_+ \to  \mathbb{R}_+$ which are linear or it is strictly increasing and strictly
convex $C^2$ function on $(0,r]$ for some $r > 0$ with $H_i(0)=H_i^{\prime}(0)=0$, such that
\begin{equation*}
	g_i^{\prime}(t)\le -\xi(t) H_i(g_i(t)),\hspace{0.15in}\forall t\ge
	0,
\end{equation*}where $\xi_i: \mathbb{R}^+ \longrightarrow \mathbb{R}^+$ are  positive nonincreasing differentiable
functions.\\ (A2) $a_i:\bar{\Omega} \longrightarrow \mathbb{R}^+$
are in $C^1 (\bar{\Omega})$ such that, for positive constant
$\delta$ and $a_0$ and for $\Gamma_1, \Gamma_2  \subset \partial
\Omega $ with means $(\Gamma_i)>0,~ i=1,2,~\inf_{x\in \bar{\Omega} }
(a_1(x)+a_2(x)) \geq \delta$ and
\begin{equation*}
	a_i=0~~~\text{or}~~~\inf_{\Gamma_i}~ a_i(x)\geq 2a_{0},~~~~~~i=1,2.
\end{equation*}
\begin{remark}
	If $a_i \neq 0,~i=1,2,$ there exist neighborhoods $w_i$ of
	$\Gamma_i,~i=1,2,$ such that $$\inf_{\bar {\Omega \bigcap w_i}}
	a_i(x) \geq a_0 > 0,~i=1,2.$$
\end{remark}

As in \cite{canti1,a1a2}, let $d=\min \{a_0, \delta \}$ and let
$\alpha _i \in C^1 (\bar{\Omega}),~i=1,2,$ be such that
\begin{equation}\label{ai}
	\begin{cases}
		\begin{array}{ll}
			0 \leq \alpha_i (x) \leq a_i(x) & \\
			\alpha_i(x)=0, & \text{if}~~ a_i(x)\leq \frac{d}{4} \\
			\alpha_i(x)=a_i(x), & \text{if}~~ a_i(x)\geq \frac{d}{2}. \\
		\end{array}
	\end{cases}
\end{equation}
\begin{lemma}
	The functions $\alpha_i,~i=1,2,$ are not identically zero and
	satisfy $\alpha_1(x)+\alpha_2(x) \geq \frac{d}{2}$.
\end{lemma}
\begin{proof}
	(1) For $x \in \Omega \cap w_i$, we have $a_i(x) \geq a_0 \geq d$,
	which implies, by \eqref{ai}, that $\alpha_i(x)=a_i(x) \geq d.$ Thus
	$\alpha_i$ is not identically zero.\\
	(2) If $a_1(x) \geq \frac{d}{2}$, then $\alpha_1(x)=a_1(x)$.
	Consequently $\alpha_1(x)+\alpha_2(x)\geq a_1(x) \geq \frac{d}{2}$.
	If $a_1 (x) < \frac{d}{2}$, then $a_2(x)>\frac{d}{2}$ which implies,
	by \eqref{ai}, $\alpha_2(x)=a_2(x)>\frac{d}{2}$. Consequently
	$\alpha_1 (x) + \alpha_2(x) >\frac{d}{2}$. This completes the proof.
\end{proof}
The existence and uniqueness of the solution of problem \eqref{main}
can be established by using the Galerkin method. We define the
"modified" energy functional of the weak solution of problem
\eqref{main}, by
\begin{equation}\label{E}
	\begin{array}{c}
		E(t)=\frac{1}{2} \int_{\Omega} u_t^2 \partial x + \frac{1}{2}
		\int_{\Omega} \left[1-a_1(x)\int_0^t g_1(s) ds - a_2(x)
		\int_0^{+\infty} g_2(s) ds \right] | \nabla u |^2 dx \\
		+\frac{1}{2}~ g_1 ~o~ \nabla u + \frac{1}{2}~ g_2~ o ~\nabla u,
	\end{array}
\end{equation}
where
\begin{equation}
	\begin{array}{l}
		g_1 ~o~ \nabla u = \int_{\Omega} a_1(x) \int_0^t g_1(t-s) |\nabla u(t)- \nabla u(s)|^2 ds ~dx \\
		g_2 ~o~ \nabla u = \int_{\Omega} a_2(x) \int_0^{+\infty} g_2 (s)
		|\nabla u(t)-\nabla u(t-s) |^2 ds~ dx \\.
	\end{array}
\end{equation}
\begin{lemma}
	The "modified" energy functional satisfies, along the solution of
	problem \eqref{main}, the following
	\begin{equation}\label{E'}
		E'(t)= -\frac{1}{2}~ g_1 (t) \int_{\Omega} |\nabla u(t)|^2 dx  +
		\frac{1}{2}~ g'_1~o~\nabla u +\frac{1}{2}~ g'_2~o~\nabla u \leq 0,
	\end{equation}
	where
	\begin{equation}
		\begin{array}{l}
			g'_1 ~o~ \nabla u = \int_{\Omega} a_1(x) \int_0^t g'_1(t-s) |\nabla u(t)- \nabla u(s)|^2 ~ds ~dx \\
			g'_2 ~o~ \nabla u= \int_{\Omega} a_2(x) \int_0^{+\infty} g'_2 (s)
			|\nabla u(t)-\nabla u(t-s)|^2~ ds~ dx \\.
		\end{array}
	\end{equation}
\end{lemma}
\begin{proof}
	By multiplying the first equation in problem \eqref{main} by $u_t$
	and integrating over $\Omega$, using integration by parts,
	hypotheses $(A1)$ and $(A2)$ and some manipulations as in
	\cite{Cabanillas,Cavalcanti1} and others, we obtain \eqref{E'} for
	regular solutions. This inequality remains valid for weak solutions
	by a simple density argument.
\end{proof}

\section{Technical lemmas}
In this section, we introduce some fundamental lemmas. These
lemmas will help us for proving our results. We use our equations
and assumptions to prove those lemmas.
\begin{lemma} \cite{Mus-2017}For $i=1,2,$ we have
	\begin{equation}\label{kmustafa}
		\int_0^1 \left(\int_{0}^{t}g_i(t-s)\left (\nabla u(s)-\nabla
		u(t-s)\right) ds\right)^{2}dx \leq C_{\alpha}(h_i \circ \nabla
		u)(t),
	\end{equation}
	where, for any  $0 <\alpha < 1$,
	\begin{equation}\label{L1:new k}
		C_{\alpha}=\int_{0}^{t}\frac{g_i^2(s)}{\alpha g_i(s)-g_i'(s)}ds
		\hspace{0.15in}\text{ and }\hspace{0.15in}h_i(t)=\alpha
		g_i(t)-g_i^{\prime}(t).
	\end{equation}
\end{lemma}
Furthermore, using the fact that $$\frac{\alpha g_i^2(s)}{\alpha
	g_i(s)-g_i'(s)} < g_i(s)$$ and recalling the Lebesgue dominated
convergence theorem, one can deduce that
\begin{equation}\label{convergence}
	\alpha C_{\alpha}=\int_{0}^{\infty}\frac{\alpha g_i^2(s)}{\alpha
		g_i(s)-g_i'(s)}ds \to 0 \text{ as }\alpha \to 0.
\end{equation}
\begin{lemma}\label{cor2}
	There exists a positive constant $M_1$ such that
	\begin{equation}\label{h}
		\begin{aligned}
			\int_{\Omega} a_2(x) \int_{t}^{+\infty} g_2(s)\vert  \nabla u
			(t)-\nabla u(t-s)\vert_2^{2} ds  \leq M_1 h_0(t),
		\end{aligned}
	\end{equation}where $h_0(t)=\int_{0}^{+\infty}g_2(t+s)\left(1+\vert \vert \nabla u_{0}(s)\vert\vert^{2}
	\right) ds$.
\end{lemma}

\begin{proof}The proof is identical to the one in
	\cite{Guesmiaetalrecent}. Indeed, we have
	\begin{equation}\label{hproof}
		\begin{aligned}
			&\int_{\Omega} a_2(x)\int_{t}^{+\infty}g_2(s) \vert \nabla
			u(t)-\nabla u(t-s) \vert_2^{2} ds \leq \\ &  2\vert \vert a_2
			\vert \vert_{\infty}  \vert \vert \nabla u(t) \vert \vert^2
			\int_{t}^{+\infty} g_2(s) ds + 2\vert \vert a_2 \vert
			\vert_{\infty}\int_{t}^{+\infty}g_2(s) \vert \vert
			\nabla u(t-s)  \vert \vert^2 ds\\
			&\leq 2 \vert \vert a_2 \vert \vert_{\infty}\sup_{s\geq 0}\vert
			\vert  \nabla u(s)  \vert \vert^2 \int_{0}^{+\infty}g_2(t+s) ds +
			2\vert \vert a_2 \vert \vert_{\infty} \int_{0}^{+\infty}g_2(t+s)
			\vert \vert
			\nabla u(-s)  \vert \vert^2 ds\\
			&\leq \frac{4 \vert \vert a_2 \vert \vert_{\infty} E(s)}{\ell}
			\int_{0}^{+\infty}g_2(t+s) ds +2 \vert \vert a_2 \vert
			\vert_{\infty}\int_{0}^{+\infty}g_2(t+s)
			\vert \vert \nabla u_{0}(s)  \vert \vert^2 ds\\
			&\leq \frac{4 \vert \vert a_2 \vert \vert_{\infty}E(0)}{\ell}
			\int_{0}^{+\infty}g_2(t+s) ds +2\vert \vert a_2 \vert
			\vert_{\infty}
			\int_{0}^{+\infty}g_2(t+s) \vert \vert \nabla u_{0}(s)  \vert \vert^2 ds\\
			&\leq M_1 \int_{0}^{+\infty}g_2(t+s)\left(1+\vert\vert \nabla
			u_{0}(s)\vert\vert^{2} \right) ds.
		\end{aligned}
	\end{equation}
	where $M_1=\max\big\{2 \vert \vert a_2 \vert \vert_{\infty},
	\frac{4 \vert \vert a_2 \vert \vert_{\infty} E(0)}{\ell}
	\big\}.$
\end{proof}

\begin{lemma}\cite{a1a2}
	Assume that $(A1)$ and $(A2)$ are hold, the functionals
	\begin{equation*}
		\psi_1
		(t)=\int_{\Omega} u u_t dx,
	\end{equation*}
	\begin{equation*}
		\psi_2 (t)=-\int_{\Omega} \alpha_1 (x) u_t \int_0^t g_1 (t-s)
		\left(  u(t)-u(s)    \right)~ ds~dx,
	\end{equation*}
	\begin{equation*}
		\psi_3 (t)=-\int_{\Omega} \alpha_2 (x) u_t \int_0^{+\infty} g_2
		(s) \left(  u(t)-u(t-s)    \right)~ ds~dx,
	\end{equation*}
	satisfies, along the solution of \eqref{main}  and for any
	$\epsilon_1, \epsilon_2, \epsilon_3 > 0$ and $\forall t \geq 0$
	the following estimates
	\begin{equation}\label{psi1}
		\begin{array}{l}
			\psi_1'(t) \leq \int_{\Omega} u_t^2 dx - \left[ 1- \epsilon_1 -
			||a_1||_{\infty} \int_0^{+\infty} g_1(s) ds - ||a_2||_{\infty}
			\int_0^{+ \infty} g_2 (s) ds \right] \int_{\Omega} |\nabla u|^2 dx\\
			~~~~~~~~~+ \frac{c C_{\alpha}}{\epsilon_1} \left(  h_1 ~o~ \nabla u
			+ h_2 ~o~ \nabla u \right),
		\end{array}
	\end{equation}
	
	\begin{equation}\label{psi2}
		\begin{array} {l}
			\psi'_2 (t) \leq -    \left[ \int_0^t g_1(s) ds- \epsilon_2 \right]
			\int_{\Omega} \alpha_1 (x) u_t^2 dx + \frac{\epsilon_3}{2}
			\int_{\Omega} |\nabla u|^2 dx - \frac{c}{\epsilon_2}  g'_1 ~o~
			\nabla u  \\ ~~~~~~~ + \frac{c C_{\alpha}}{\epsilon_3} \left( h_1
			~o~\nabla u+ h_2 ~o~\nabla u \right),
		\end{array}
	\end{equation}
	
	\begin{equation}\label{psi3}
		\begin{array} {l}
			\psi'_3 (t) \leq - \left[ \int_0^{+\infty} g_2(s) ds- \epsilon_2
			\right] \int_{\Omega} \alpha_2 (x) u_t^2 dx + \frac{\epsilon_3}{2}
			\int_{\Omega} |\nabla u|^2 dx - \frac{c}{\epsilon_2}  g'_2 ~o~
			\nabla u  \\ ~~~~~~~ + \frac{c C_{\alpha}}{\epsilon_3} \left( h_1
			~o~\nabla u+ h_2 ~o~\nabla u \right).
		\end{array}
	\end{equation}
	
\end{lemma}
\begin{lemma}\label{lemma3}
	Assume that $(A1)$ and $(A2)$ are hold, then for a suitable choice
	of  $ N, M, m > 0$ and for all $t \geq 0$, the functional
	$$\mathcal{L}(t):=NE(t)+M \psi_{1}(t)+ {\psi_{2}}(t)+{\psi_{3}}(t)$$
	satisfies for any $t>0$
	\begin{equation}\label{E:main0}
		\mathcal{L}^{\prime}(t) \le - m E(t)+ \frac{1}{4} \bigg(g_1 \circ
		\nabla u + g_2 \circ \nabla u)\bigg)(t).
	\end{equation}
\end{lemma}
\begin{proof}Let, for $t_0 >0$ fixed, $g_0:= \min \{ \int_0^{t_0} g_1(s) ds, \int_0^{+\infty} g_2(s)
	ds\}$. Then direct differentiation of $\mathcal{L}$ and using
	\eqref{E'}, \eqref{psi1}, \eqref{psi2} and \eqref{psi3} leads to
	\begin{equation*}
		\begin{aligned}
			&\mathcal{L}^{\prime}(t) \leq
			\bigg(\frac{N}{2}-\frac{c}{\varepsilon_2}\bigg)\bigg((g_1' \circ
			\nabla u)(t)+(g_2' \circ \nabla u)(t)\bigg)\\&-\int_{\Omega}
			\bigg[(g_0-\varepsilon_2)(\alpha_1+\alpha_2)-M\bigg]u_t^2 dx\\
			&\times \bigg(\frac{c C_{\alpha}}{\varepsilon_3}+\frac{Mc
				C_{\alpha}}{\varepsilon_1}\bigg)\bigg((h_1 \circ \nabla u)(t)+(h_2
			\circ \nabla u)(t)\bigg)\\&-\bigg[
			(\ell-\varepsilon_1)M-\varepsilon_3 \bigg]\int_{\Omega} \vert \nabla
			u \vert^2 dx,~t\geq t_0.
		\end{aligned}
	\end{equation*}
	Recalling that $g_i'=\alpha g_i-h_i$ where $i=1,2.$ We obtain
	\begin{equation*}
		\begin{aligned}
			\mathcal{L}^{\prime}(t) &\leq \alpha
			\bigg(\frac{N}{2}-\frac{c}{\varepsilon_2}\bigg)\bigg(g_1 \circ
			\nabla u + g_2 \circ \nabla u \bigg)(t)\\
			&-\bigg(\frac{N}{2}-\frac{c}{\varepsilon_2}\bigg)\bigg(h_1 \circ
			\nabla u + h_2 \circ \nabla u \bigg)(t)\\
			&-\int_{\Omega}
			\bigg[(g_0-\varepsilon_2)(\alpha_1+\alpha_2)-M\bigg]u_t^2 dx\\
			&\times \bigg(\frac{c C_{\alpha}}{\varepsilon_3}+\frac{Mc
				C_{\alpha}}{\varepsilon_1}\bigg)\bigg(h_1 \circ \nabla u)(t)+(h_2
			\circ \nabla u)(t)\bigg)\\&-\bigg[
			(\ell-\varepsilon_1)M-\varepsilon_3 \bigg]\int_{\Omega} \vert \nabla
			u \vert^2 dx,~t\geq t_0.
		\end{aligned}
	\end{equation*}
	Simplifying the above estimate, we arrive at
	\begin{equation*}
		\begin{aligned}
			\mathcal{L}^{\prime}(t) &\leq \alpha
			\bigg(\frac{N}{2}-\frac{c}{\varepsilon_2}\bigg)\bigg(g_1 \circ
			\nabla u + g_2 \circ \nabla u\bigg)(t)\\
			&-\bigg[\frac{N}{2}-\frac{c}{\varepsilon_2}-C_{\alpha}
			\big(\frac{c}{\varepsilon_3}+\frac{Mc}{\varepsilon_1} \big)
			\bigg]\bigg(h_1 \circ \nabla u + h_2 \circ \nabla u\bigg)(t)\\
			&-\int_{\Omega}
			\bigg[(g_0-\varepsilon_2)(\alpha_1+\alpha_2)-M\bigg]u_t^2 dx\\
			&-\bigg[ (\ell-\varepsilon_1)M-\varepsilon_3 \bigg]\int_{\Omega}
			\vert \nabla u \vert^2 dx,~t\geq t_0.
		\end{aligned}
	\end{equation*}
	By using the fact that $(\alpha_1+\alpha_2) \geq \frac{d}{2}$ and
	choosing
	$$\varepsilon_1=\frac{\ell}{2}, ~~\varepsilon_2=\frac{g_0}{2}~~M =\frac{dg_0}{8}'~~\varepsilon_3 < \frac{d\ell^2 d g_0}{16}$$
	we get
	
	$$ (\ell-\varepsilon_1)M-\varepsilon_3 \geq 4 (1- \ell)$$
	As a consequence of \eqref{convergence},  there exists $0<\alpha_0
	<1$ such that if $\alpha < \alpha_0$, then
	\begin{equation}\label{con:r1:4a}
		\alpha
		C_{\alpha}<\frac{1}{8\bigg[\frac{c}{\varepsilon_3}+\frac{Mc}{\varepsilon_1}
			\bigg]}.
	\end{equation}It is clear that \eqref{con:r1:4a} gives
	\begin{equation}\label{con:r1:4B}
		C_{\alpha}\bigg[ \frac{N}{2}-\frac{c}{\varepsilon_2}-C_{\alpha}
		\big(\frac{c}{\varepsilon_3}+\frac{Mc}{\varepsilon_1}
		\big)\bigg]<\frac{N}{2}.
	\end{equation}
	Choosing $\alpha=\frac{1}{2N}<\alpha_0,$, then for some $m >0$, we
	have
	\begin{equation*}
		\mathcal{L}^{\prime}(t) \le - m E(t)+ \frac{1}{4} \bigg(g_1 \circ
		\nabla u + g_2 \circ \nabla u)\bigg)(t).
	\end{equation*}Then \eqref{E:main0} is established. Moreover, one can choose  $N$ large enough so that $\mathcal{L}\sim
	E$.
\end{proof}

\begin{lemma}\label{lemma4a} For all $t\in
	\mathbb{R}^+$ and  fixed positive constants $m_0$ , we have the
	following estimates
	\begin{equation}\label{E1:r13:St}
		\int_{0}^{t} E(s) ds < m_0 \left(1+ \int_{0}^{t} h_0 (s) ds\right).
	\end{equation} where $h_0$ is defined in \eqref{h}.
\end{lemma} \begin{proof} The proof of this lemma can be established by following the same arguments in
	\cite{al2020general,Guesmiaetalrecent}. \end{proof}

\begin{lemma}\label{L2:St}
	If  $(A1)-(A2)$ are satisfied, then we have, for all $t > 0$ and for
	$i=1,2$,  the following estimates
	\begin{equation}\label{E1c:L2:St0}
		\int_{\Omega} a_i(x) \int_{0}^{t}g_i(s) \vert \nabla u(t)-\nabla
		u(t-s) \vert^{2}ds dx \le \frac{1}{\Lambda_i(t)}
		{H_i}^{-1}\left(\frac{ \Lambda_i(t)\mu_i(t)}{ \xi_i (t)}\right)
	\end{equation}
	where $\Lambda_i(t)=\frac{\lambda_i}{1+\int_{0}^{t} h_0(s)ds}$; $
	\lambda_i \in (0,1)$, $H_i$ and $\xi_i$ are introduced in $(A1)$,
	and
	\begin{equation}\label{E3d:r13:St}
		\mu_i(t):=-\int_{\Omega} a_i(x) \int_{0}^{t} g_i'(s) \vert
		\nabla(t)-\nabla(t-s)\vert^{2}ds dx\le -c E^{\prime}(t),
	\end{equation}.
\end{lemma}
\begin{proof}
	To establish \eqref{E1c:L2:St0},  we introduce the following
	functional
	\begin{equation}\label{lamda}
		\chi_i (t):=\Lambda_i (t)  \int_{\Omega} a_i(x) \int_{0}^{t}\vert
		\nabla u(t)-\nabla u(t-s)\vert^{2}ds dx.
	\end{equation}
	Then, using the fact that $E$ is nonincreasing and \eqref{E} to
	get
	\begin{equation}\label{new}
		\begin{aligned}
			&\chi_i (t)\leq 2\Lambda_i(t)\bigg(\int_{\Omega} a_i(x) \int_{0}^{t}\vert \nabla u(t)\vert^2 + \int_{\Omega} a_i(x) \int_{0}^{t}\vert \nabla  u(t-s)\vert^{2}dsdx\bigg).\\
			&\hspace{0.4in}\le \frac{4 \Lambda_i(t) \vert \vert a_i \vert \vert_{\infty} }{\ell }\bigg(\int_{0}^{t} \big(E(t)+E(t-s)\big)ds\bigg)\\
			&\hspace{0.4in}\le \frac{8 \Lambda_i(t) \vert \vert a_i \vert \vert_{\infty}}{\ell } \int_{0}^{t} E(s) ds\\
			&\hspace{0.4in} < + \infty.
		\end{aligned}
	\end{equation}
	Thus,  $\lambda_i$ can be chosen so small so that, for all $t > 0$,
	\begin{equation}\label{E2:r13:St1}
		\chi_i (t)<1.
	\end{equation}
	Without loss of the generality, for all $t > 0$, we assume that
	$\xi_i(t)> 0$, otherwise we get an exponential decay from
	\eqref{E:main0}. The use of Jensen's inequality and using
	\eqref{E3d:r13:St} and \eqref{E2:r13:St1} gives
	\begin{equation}\label{p6:inv}
		\begin{aligned}
			&\mu_1(t)=\frac{1}{\Lambda_i(t)} \int_{0}^{t}\Lambda_i (t)(-g_i'(s))\int_{\Omega}{   a_i(x) \vert \nabla(t)-\nabla(t-s)\vert}^{2}dxds\\
			&\hspace{0.3in}\ge \frac{1}{\Lambda_i(t)}\int_{0}^{t}\Lambda_i (t)\xi_1(s) H_i(g_i(s))\int_{\Omega}{ a_i(x) \vert \nabla(t)-\nabla(t-s)\vert}^{2}dxds\\
			&\hspace{0.3in}\ge \frac{\xi_i(t)}{\Lambda_i(t)}\int_{0}^{t}H_i(\Lambda_i(t)g_i(s))\int_{\Omega}{ a_i(x)\vert \nabla(t)-\nabla(t-s)\vert}^{2}dxds\\
			&\hspace{0.3in}\ge \frac{\xi_i(t)}{\Lambda_i (t)}H_i\biggl(\Lambda_i (t)\int_{0}^{t}g_i(s)\int_{\Omega}a_i(x) {\vert \nabla(t)-\nabla(t-s)\vert}^{2}dxds\biggr)\\
			&\hspace{0.3in}=\frac{
				\xi_1(t)}{\Lambda_i}\overline{H_i}\biggl(\Lambda_i(t)\int_{0}^{t}g_i(s)\int_{\Omega}
			a_1(x){\vert \nabla u(t)-\nabla u(t-s)\vert}^{2}dxds\biggr),
		\end{aligned}
	\end{equation}
	Hence, \eqref{E1c:L2:St0} is established.
\end{proof}
\section{Decay result}\label{Decay}
In this section, we state, prove our main result and provide some
example to illustrate our decay results. Let us start introducing
some functions and then establishing several lemmas needed for the
proof of our main result. Now, for $i=1,2,$ let us take
\begin{equation}\label{minmax}
	\xi= \min \{\xi_i \}, ~~\mu= \max \{\mu_i \},~~\lambda= \max \{\lambda_i
	\}~~G=\left(H_1^{-1}+H_2^{-1}\right)^{-1}.
\end{equation}
As in \cite{Guesmiaetalrecent}, we introduce the following
functions:
\begin{equation}\label{G1}
	G_{1}(t):=\int_{t}^{r}\frac{1}{s G^{\prime}(s)}ds,~~G_5(t)=G_1^{-1}\bigg(c_1\int_0^t
	\xi(s)ds\bigg),
\end{equation}
\begin{equation}\label{G234}
	G_{2}(t)=t G^{\prime}(t),\quad G_3(t)=t
	(G')^{-1}(t), \quad G_4(t)=G_3^{*}(t).
\end{equation}
One can easily verify that $G_1$ is decreasing function over
$(0,r]$, and $G_{2}, G_3$ and $G_4$ be convex and increasing
functions on $(0,r]$. Further,  we introduce the class $S$ of
functions $\vartheta: \mathbb{R}_+ \rightarrow (0,r]$ satisfying for
fixed $c_1, c_2 > 0$,
\begin{equation}\label{xi}
	\vartheta\in C^1(\mathbb{R}_+),~~~\vartheta' \leq 0,
\end{equation}
and
\begin{equation}\label{dif}
	c_2 G_4\left[\frac{c}{d} \Lambda(t)
	h_0(t)\right]\leq c_1  \left(
	G_2\bigg(\frac{G_5(s)}{\vartheta(s)}\bigg)-\frac{G_2\left(G_5(t)\right)}{\vartheta(t)}\right),
\end{equation}
where $h_2$ and $q$ will be defined later in the proof of our main
result.
\begin{theorem}\label{main:th2017:1a}
	Assume that $(A1)-(A2)$  hold, then there exists a strictly positive
	constant $C$ such that the solution of \eqref{main} satisfies, for
	all $t \geq 0$,
	\begin{equation}\label{decay1e}
		E(t)\leq \frac{C G_5(t)}{\vartheta(t)\Lambda(t)},
	\end{equation}where $G_5$ and $\vartheta$ are defined in \eqref{G1} and \eqref{xi}
	respectively.
\end{theorem}
\begin{proof}
	Using  \eqref{h}, \eqref{E:main0} and \eqref{E1c:L2:St0}, then for
	some positive constant $m$,  and any $t \geq 0$, we get
	\begin{equation}\label{e1s4b}
		\begin{aligned}
			L^{\prime}(t) \le -m
			E(t)+\frac{c}{\Lambda(t)}G^{-1}\left(\frac{\Lambda(t)
				\mu(t)}{\xi(t)}\right)+ c h_0(t).
		\end{aligned}
	\end{equation}
	Without lose of generality, one can assume that $E(0) > 0$. For
	$\varepsilon_{0} < r $, let the functional $\mathcal{F}$ defined by
	$$\mathcal{F}(t):=G^{\prime}\left(\varepsilon_{0}\frac{E(t)\Lambda(t)}{E(0)}\right)L(t),$$
	which satisfies $\mathcal{F} \sim E$. By noting that
	$G^{\prime\prime}\geq0,$ $\Lambda'\leq0$ and $E'\leq0$, we get
	\begin{equation}\label{p6:F10b}
		\begin{aligned}
			&\mathcal{F}^{\prime}(t)=\varepsilon_{0}\frac{(qE)^{\prime}(t)}{E(0)}G^{\prime\prime}\left(\varepsilon_{0}\frac{E(t)\Lambda(t)}{E(0)}\right)L(t)+G^{\prime}\left(\varepsilon_{0}\frac{E(t)\Lambda(t)}{E(0)}\right)L^{\prime}(t)\\
			&\hspace{0.3in}\le -m
			E(t)G^{\prime}\left(\varepsilon_{0}\frac{E(t)\Lambda(t)}{E(0)}\right)+\frac{c}{\Lambda(t)}
			G^{\prime}\left(\varepsilon_{0}\frac{E(t)\Lambda(t)}{E(0)}\right)G^{-1}\left(\frac{\Lambda(t)
				\mu(t)}{\xi(t)}\right)\\&\hspace{0.3in}+c h_0(t)
			G^{\prime}\left(\varepsilon_{0}\frac{E(t)\Lambda(t)}{E(0)}\right).
		\end{aligned}
	\end{equation}
	Let $G^{*}$ be the convex conjugate of $G$ in the sense of Young
	(see \cite{arnol2013mathematical}), then
	\begin{equation}\label{p6:conj0b}
		G^{*}(s)=s(G^{\prime})^{-1}(s)-G\left[(G^{\prime})^{-1}(s)\right],\hspace{0.1in}\text{if}\hspace{0.05in}s\in
		(0,G^{\prime}(r)]
	\end{equation}and $G^{*}$ satisfies the following generalized Young inequality
	\begin{equation}\label{p6:young0b}
		A B\le G^{*}(A)+G(B),\hspace{0.15in}\text{if}\hspace{0.05in}A\in
		(0,G^{\prime}(r)],\hspace{0.05in}B\in(0,r].
	\end{equation}So, with $A=G^{\prime}\left(\varepsilon_{0}\frac{E(t)\Lambda(t)}{E(0)}\right)$
	and $B=G^{-1}\left(\frac{\Lambda(t) \mu(t)}{\xi(t)}\right)$ and  using \eqref{p6:F10b}-\eqref{p6:young0b}, we arrive at
	\begin{equation}\label{E:m:xi10b}
		\begin{aligned}
			&\mathcal{F}^{\prime}(t)\le -m
			E(t)G^{\prime}\left(\varepsilon_{0}\frac{E(t)\Lambda(t)}{E(0)}\right)+\frac{c}{\Lambda(t)}
			G^{*}\left(G^{\prime}\left(\varepsilon_{0}\frac{E(t)\Lambda(t)}{E(0)}\right)\right)+c
			\left(\frac{ \mu(t)}{\xi(t)}\right)\\&\hspace{0.3in}+c h_0(t) G^{\prime}\left(\varepsilon_{0}\frac{E(t)\Lambda(t)}{E(0)}\right)\\
			&\hspace{0.3in}\le  -m
			E(t)G^{\prime}\left(\varepsilon_{0}\frac{E(t)\Lambda(t)}{E(0)}\right)+c\varepsilon_{0}\frac{E(t)}{E(0)}G^{\prime}\left(\varepsilon_{0}\frac{E(t)\Lambda(t)}{E(0)}\right)+c
			\left(\frac{ \mu(t)}{\xi(t)}\right)\\&\hspace{0.3in}+c h_0(t)
			G^{\prime}\left(\varepsilon_{0}\frac{E(t)\Lambda(t)}{E(0)}\right).
		\end{aligned}
	\end{equation}
	So, multiplying \eqref{E:m:xi10b} by $\xi(t)$ and using
	\eqref{E3d:r13:St} and the fact that
	$\varepsilon_{0}\frac{E(t)\Lambda(t)}{E(0)}<r$, gives
	\begin{equation*}
		\begin{aligned}
			&\xi(t) \mathcal{F}^{\prime}(t)\le -m  \xi(t)
			E(t)G^{\prime}\left(\varepsilon_{0}\frac{E(t)\Lambda(t)}{E(0)}\right)+c
			\xi(t)\varepsilon_{0}\frac{E(t)}{E(0)}G^{\prime}\left(\varepsilon_{0}\frac{E(t)\Lambda(t)}{E(0)}\right)\\&\hspace{0.3in}+c
			\mu(t)\Lambda(t)+c \xi(t)h_0(t)G^{\prime}\left(\varepsilon_{0}\frac{E(t)\Lambda(t)}{E(0)}\right)\\
			&\hspace{0.3in}\le-\varepsilon_{0}(\frac{m  E(0)}{\varepsilon_{0}}
			-c )
			\xi(t)\frac{E(t)}{E(0)}G^{\prime}\left(\varepsilon_{0}\frac{E(t)\Lambda(t)}{E(0)}\right)-c
			E^{\prime}(t)+c \xi(t)
			h_0(t)G^{\prime}\left(\varepsilon_{0}\frac{E(t)\Lambda(t)}{E(0)}\right).
		\end{aligned}
	\end{equation*}Consequently, recalling the definition of $G_2$ and choosing $\varepsilon_{0}$ so that $k=(\frac{m  E(0)}{\varepsilon_{0}}
	-c)>0$, we obtain, for all $t \in \mathbb{R}_+$,
	\begin{equation}\label{p6:main40b}
		\begin{aligned}
			\mathcal{F}_{1}^{\prime}(t)&\le -k
			\xi(t)\left(\frac{E(t)}{E(0)}\right)G^{\prime}\left(\varepsilon_{0}\frac{E(t)\Lambda(t)}{E(0)}\right)+c
			\xi(t)
			h_0(t)G^{\prime}\left(\varepsilon_{0}\frac{E(t)\Lambda(t)}{E(0)}\right)\\&=-k\frac{\xi(t)}{\Lambda(t)}
			G_{2}\left(\frac{E(t)\Lambda(t)}{E(0)}\right)+c \xi(t)
			h_0(t)G^{\prime}\left(\varepsilon_{0}\frac{E(t)\Lambda(t)}{E(0)}\right),
		\end{aligned}
	\end{equation}
	where $\mathcal{F}_{1}=\xi \mathcal{F}+c E \sim E$ and satisfies for
	some $\alpha_{1},\alpha_{2}>0.$
	\begin{equation}\label{p6:equiv20b}
		\alpha_{1}\mathcal{F}_1(t)\le E(t)\le \alpha_{2}\mathcal{F}_1(t).
	\end{equation}
	Since
	$G^{\prime}_{2}(t)=G^{\prime}(t)+t G^{\prime\prime}(t)$ and $G$ is
	strictly increasing and strictly convex on $(0,r],$ we find that
	$G_{2}^{\prime}(t), G_{2}(t)>0$ on $(0,r].$ Using the general Young
	inequality \eqref{p6:young0b} on the last term in \eqref{p6:main40b}
	with
	$A=G^{\prime}\left(\varepsilon_{0}\frac{E(t)\Lambda(t)}{E(0)}\right)$
	and $B=[\frac{c}{d}h_0 (t)]$,  we have for $d>0$
	\begin{equation}\label{essa1a}
		\begin{aligned}
			c h_0
			(t)G^{\prime}\left(\varepsilon_{0}\frac{E(t)\Lambda(t)}{E(0)}\right)&=\frac{d}{\Lambda(t)}
			\left[\frac{c}{d }\Lambda(t) h_0(t)\right] \bigg(G'
			\left(\varepsilon_{0}\frac{E(t)\Lambda(t)}{E(0)}\right)\bigg)\\&\leq
			\frac{d}{\Lambda(t)}
			G_3\bigg(G'\left(\varepsilon_{0}\frac{E(t)\Lambda(t)}{E(0)}\right)\bigg)+\frac{d}{\Lambda(t)}G_3^*\left[\frac{c}{d
			}\Lambda(t) h_0(t)\right]\\&\leq \frac{d}{\Lambda(t)}
			\left(\varepsilon_{0}\frac{E(t)\Lambda(t)}{E(0)}\right)\left(G'\left(\varepsilon_{0}\frac{E(t)\Lambda(t)}{E(0)}\right)\right)+\frac{d}{\Lambda(t)}
			G_4\left[\frac{c}{d }\Lambda(t) h_0(t)\right]\\&\leq
			\frac{d}{\Lambda(t)}
			G_2\left(\varepsilon_{0}\frac{E(t)\Lambda(t)}{E(0)}\right)+\frac{d}{\Lambda(t)}
			G_4\left[\frac{c}{d }\Lambda(t) h_0(t)\right].
		\end{aligned}
	\end{equation}
	Now, combining \eqref{p6:main40b} and \eqref{essa1a} and choosing
	$d$ small enough so that $k_0=(k-d)>0$, we arrive at
	\begin{equation}\label{R'}
		\begin{aligned}
			& \mathcal{F}_{1}^{\prime}(t) \leq - k \frac{\xi (t)}{\Lambda(t)} G_2 \left(\varepsilon_{0}\frac{E(t)\Lambda(t)}{E(0)}\right) + \frac{d \xi(t)}{\Lambda(t)}
			G_2\left(\varepsilon_{0}\frac{E(t)\Lambda(t)}{E(0)}\right)+\frac{d
				\xi(t)}{\Lambda(t)}
			G_4\left[\frac{c}{d }\Lambda(t) h_0(t)\right]\\
			&\hspace{0.1in}\leq  - k_1 \frac{\xi (t)}{\Lambda(t)} G_2
			\left(\varepsilon_{0}\frac{E(t)\Lambda(t)}{E(0)}\right)+\frac{d
				\xi(t)}{\Lambda(t)}G_4\left[\frac{c}{d }\Lambda(t) h_0(t)\right].
		\end{aligned}
	\end{equation}
	Using the equivalent property in \eqref{p6:equiv20b} and the
	increasing of $G_2$, we have
	\begin{equation*}
		G_2 \left(\varepsilon_{0}\frac{E(t)\Lambda(t)}{E(0)}\right)\geq G_2
		\bigg(d_{0} \mathcal{F}_{1}(t)\Lambda(t)\bigg).
	\end{equation*}
	Letting $\mathcal{F}_{2}(t):=d_{0} \mathcal{F}_{1}(t)\Lambda(t)$ and
	recalling $\Lambda'\leq0$, then we arrive at, for some $c_1, c_2
	>0$,
	\begin{equation}\label{4.36}
		\mathcal{F}_{2}'(t)\leq -c_1\xi (t)G_2(\mathcal{F}_{2}(t))+c_2
		\xi(t)G_4\left[\frac{c}{d }\Lambda(t) h_0(t)\right].
	\end{equation}%
	Since $d_{0}\Lambda(t)$ is  nonincreasing. Using the equivalent
	property $\mathcal{F}_{1}\sim E$ implies that there exists $b_0 > 0$
	such that $\mathcal{F}_{2}(t)\geq b_{0} E(t)\Lambda(t)$. Let $t\in
	\mathbb{R}_+$ and $\vartheta(t)$ satisfying \eqref{xi} and
	\eqref{dif}.\\If
	\begin{equation}\label{decay000}
		b_0\Lambda(t) E(t) \leq 2 \frac{G_5(t)}{\vartheta(t)},
	\end{equation}
	then, we have
	\begin{equation}\label{decay1}
		E(t) \leq \frac{2 }{b_0} \frac{G_5(t)}{\vartheta(t)\Lambda(t)}.
	\end{equation}
	If
	\begin{equation}\label{decay000}
		b_0\Lambda(t) E(t) > 2  \frac{G_5(t)}{\vartheta(t)}.
	\end{equation}
	Then, for any $0\leq s \leq t$, we get
	\begin{equation}\label{decay000}
		b_0q(s) E(s) > 2  \frac{G_5(t)}{\vartheta(t)},
	\end{equation}
	since, $\Lambda(t) E(t)$ is nonincreasing function. Therefore, for
	any $0\leq s \leq t$, we have
	\begin{equation}\label{psi1}
		\mathcal{F}_{2}(s) > 2 \frac{G_5(t)}{\vartheta(t)}.
	\end{equation} Using \eqref{psi1}, recalling the definition of $G_2$,  the fact that $G_2$ is
	convex and $0< \vartheta\leq 1$, we have, for any $0 \leq s \leq t$
	and $0 < \epsilon_1 \leq 1$,
	\begin{equation}\label{g2a}
		\begin{aligned}
			&G_2\bigg(\epsilon_1 \vartheta(s)\mathcal{F}_{2}(s)-\epsilon_1
			G_5(s)\bigg) \leq \epsilon_1
			\vartheta(s)G_2\bigg(\mathcal{F}_{2}(s)-
			\frac{G_5(s)}{\vartheta(s)}\bigg)\\
			&\leq \epsilon_1 \vartheta(s)\mathcal{F}_{2}(s)
			G'\bigg(\mathcal{F}_{2}(s)-
			\frac{G_5(s)}{\vartheta(s)}\bigg)-\epsilon_1
			\vartheta(s)\frac{G_5(s)}{\vartheta(s)}G'\bigg(\mathcal{F}_{2}(s)-
			\frac{G_5(s)}{\vartheta(s)}\bigg)\\
			&\leq \epsilon_1 \vartheta(s)\mathcal{F}_{2}(s)
			G'\bigg(\mathcal{F}_{2}(s)\bigg)-\epsilon_1
			\vartheta(s)\frac{G_5(s)}{\vartheta(s)}G'\bigg(\frac{G_5(s)}{\vartheta(s)}\bigg).
		\end{aligned}
	\end{equation}
	Now, we let
	\begin{equation}\label{F3}
		\mathcal{F}_{3}(s)=\epsilon_1 \vartheta(s)\mathcal{F}_{2}(s)-\epsilon_1
		G_5(s),
	\end{equation} where $\epsilon_1$ small enough so that $\mathcal{F}_{3}(0)\leq 1$. Then \eqref{g2a}
	becomes, for any $0 \leq s \leq t$,
	\begin{equation}\label{4.37}
		\begin{aligned}
			&G_2\bigg(\mathcal{F}_{3}(s)\bigg) \leq \epsilon_1
			\vartheta(s)G_2\bigg(\mathcal{F}_{2}(s)\bigg)- \epsilon_1
			\vartheta(s)G_2\bigg(\frac{G_5(s)}{\vartheta(s)}\bigg).
		\end{aligned}
	\end{equation}
	Further, we have
	\begin{equation}\label{g2}
		\begin{aligned}
			\mathcal{F}'_{3}(t)=\epsilon_1 \vartheta'(t)\mathcal{F}_{2}(t)
			+\epsilon_1\vartheta(s)\mathcal{F}'_{2}(t)-\epsilon_1 G_5'(t).
		\end{aligned}
	\end{equation}
	Since $\vartheta' \leq 0$ and using \eqref{4.36}, then for any $0
	\leq s \leq t$, $0 < \epsilon_1 \leq 1$,  we obtain
	\begin{equation}\label{g2}
		\begin{aligned}
			&\mathcal{F}'_{3}(t)\leq
			\epsilon_1\vartheta(t)\mathcal{F}'_{2}(t)-\epsilon_1 G_5'(t)\\
			&\leq -c_1\epsilon_1 \xi (t)
			\vartheta(t)G_2(\mathcal{F}_{2}(t))+c_2\epsilon_1 \xi (t)
			\vartheta(s)G_4\left[\frac{c}{d }\Lambda(t)
			h_0(t)\right]-\epsilon_1G_5'(t).
		\end{aligned}
	\end{equation}
	Then, using \eqref{dif} and \eqref{4.37}, we get
	\begin{equation}\label{f3}
		\begin{aligned}
			&\mathcal{F}'_{3}(t) \leq -c_1 \xi (t)
			G_2(\mathcal{F}_{3}(t))+c_2\epsilon_1 \xi
			(t)\vartheta(t)G_4\left[\frac{c}{d }\Lambda(t)
			h_0(t)\right]\\&-c_1\epsilon_1 \xi
			(t)\vartheta(t)G_2\bigg(\frac{G_5(t)}{\vartheta(t)}\bigg)-\epsilon_1G_5'(t).
		\end{aligned}
	\end{equation}
	From the definition of $G_1$ and $G_5$, we have
	\begin{equation*}
		G_1\left(G_5(s)\right)=c_1\int_0^s \xi(\tau) d\tau,
	\end{equation*}
	hence,
	\begin{equation}\label{psi'}
		G_5'(s)=-c_1 \xi(s) G_2\left(G_5(s)\right).
	\end{equation}
	Now, we have
	\begin{equation}\label{coll}
		\begin{aligned}
			&c_2\epsilon_1 \xi (t)\vartheta(t)G_4\left[\frac{c}{d }\Lambda(t)
			h_0(t)\right]-c_1\epsilon_1 \xi
			(t)\vartheta(t)G_2\bigg(\frac{G_5(t)}{\vartheta(t)}\bigg)-\epsilon_1G_5'(t)\\&=c_2\epsilon_1
			\xi (t)\vartheta(t)G_4\left[\frac{c}{d }\Lambda(t)
			h_0(t)\right]-\epsilon_1 \xi
			(t)\vartheta(t)G_2\bigg(\frac{G_5(t)}{\vartheta(t)}\bigg)+c\epsilon_1\xi(t)G_2\left(G_5(t)\right)\\&=\epsilon_1
			\xi (t)\vartheta(t)\bigg(c_2 G_4\left[\frac{c}{d }\Lambda(t)
			h_0(t)\right]-c_1G_2\bigg(\frac{G_5(t)}{\vartheta(t)}\bigg)+c_1\frac{G_2\left(G_5(t)\right)}{\vartheta(t)}\bigg).
		\end{aligned}
	\end{equation}
	Then, according to \eqref{dif}, we get
	$$\epsilon_1 \xi
	(t)\vartheta(t)\bigg(c_2G_4\left[\frac{c}{d }\Lambda(t)
	h_0(t)\right]-c_1G_2\bigg(\frac{G_5(t)}{\vartheta(t)}\bigg)-c_1\frac{G_2\left(G_5(t)\right)}{\vartheta(t)}\bigg)\leq
	0 $$ Then \eqref{f3} gives
	\begin{equation}\label{f30}
		\begin{aligned}
			&\mathcal{F}'_{3}(t) \leq -c_1 \xi (t) G_2(\mathcal{F}_{3}(t)).
		\end{aligned}
	\end{equation}
	Thus  from \eqref{f30} and the definition of $G_1$ and $G_2$ in
	\eqref{G1} and \eqref{G234}, we obtain
	\begin{equation}\label{end1}
		\bigg(G_1\left(\mathcal{F}_{3}(t)\right)\bigg)'\geq c_1
		\xi(t).
	\end{equation}
	Integrating  \eqref{end1} over $[0, t]$, we get
	\begin{equation}\label{end2}
		G_1\left(\mathcal{F}_{3}(t)\right)\geq c_1\int_0^t
		\xi(s)ds+G_1\left(\mathcal{F}_{3}(0)\right).
	\end{equation}
	Since $G_1$ is decreasing, $\mathcal{F}_{3}(0)\leq 1$ and
	$G_1(1)=0$, then
	
	\begin{equation}\label{end3}
		\mathcal{F}_{3}(t)\leq G_1^{-1}\bigg(c_1\int_0^t
		\xi(s)ds\bigg)=G_5(t).
	\end{equation}
	Recalling that $ \mathcal{F}_{3}(t)=\epsilon_1
	\vartheta(t)\mathcal{F}_{2}(t)-\epsilon_1 G_5(t)$, we have
	
	\begin{equation}\label{end4}
		\mathcal{F}_{2}(t)\leq \frac{(1+\epsilon_1)}{\epsilon_1}
		\frac{G_5(t)}{\vartheta(t)},
	\end{equation}
	Similarly, recall that $\mathcal{F}_{2}(t):=d_{0}
	\mathcal{F}_{1}(t)\Lambda(t)$, then
	\begin{equation}\label{end5}
		\mathcal{F}_{1}(t)\leq \frac{(1+\epsilon_1)}{d_0\epsilon_1}
		\frac{G_5(t)}{\vartheta(t)\Lambda(t)},
	\end{equation}
	Since $\mathcal{F}_{1}\sim E$, then for some $b>0$, we have
	$E(t)\leq b \mathcal{F}_{1}$; which gives
	\begin{equation}\label{decay2}
		E(t)\leq \frac{b(1+\epsilon_1)}{d_0\epsilon_1}
		\frac{G_5(t)}{\vartheta(t)\Lambda(t)},
	\end{equation}
	From \eqref{decay1} and \eqref{decay2}, we obtain the following
	estimate
	\begin{equation}\label{decay3}
		E(t)\leq c_3 \bigg(\frac{G_5(t)}{\vartheta(t)\Lambda(t)}\bigg),
	\end{equation}
	where $c_3=\max\{\frac{2}{b_0},
	\frac{b(1+\epsilon_1)}{d_0\epsilon_1}\}.$
\end{proof}
\section{Examples}
Let $g(t)=\frac{a}{(1+t)^\beta}$, where $\beta
>1$ and $0<a<\beta-1$ so that $(A1)$ is satisfied. In this
case $\xi(t)=\beta a^{\frac{-1}{\beta}}$ and
$G(t)=t^{\frac{\beta+1}{\beta}}$. Then, there exist positive
constants $a_i (i = 0,...,3)$ depending only on $a, \beta$ such that
\begin{equation}\label{gs}\begin{aligned}
		&G_4(t)=a_0 t^{\frac{\beta+1}{\beta}},~~~G_2(t)=a_1
		t^{\frac{\beta+1}{\beta}},~~~~G_1(t)=a_2
		(t^{\frac{-1}{\beta}}-1),~~G_5(t)=(a_3t+1)^{-\beta}.
	\end{aligned}
\end{equation}
We  will discuss two cases:\\Case 1: if
\begin{equation}\label{r}
	m_0 (1+t)^{r} \leq 1+\vert \vert \nabla u_{0}\vert \vert^2 \leq m_1
	(1+t)^{r}
\end{equation}
where $0 < r < \beta-1$ and $m_0 , m_1 > 0$, then we have, for some
positive constants $a_i (i = 4,...,7)$ depending only on $a, \beta,
m_0, m_1, r$,  the following:
\begin{equation}\label{gs0}\begin{aligned}
		& a_4 (1+t)^{-\beta+1+r}\leq h_0(t) \leq a_5 (1+t)^{-\beta+1+r},
	\end{aligned}
\end{equation}
\begin{equation}\label{gs10}\begin{aligned}
		&\frac{\lambda}{\Lambda(t)} \geq a_6\left\{%
		\begin{array}{ll}
			1+\ln(1+t), & \hbox{$\beta-r=2$;} \\
			2, & \hbox{$\beta -r> 2$;} \\
			(1+t)^{-\beta+r+2}, & \hbox{$1< \beta -r <2$ .} \\
		\end{array}%
		\right.
\end{aligned}\end{equation}
\begin{equation}\label{gs20}\begin{aligned}
		&\frac{\lambda}{\Lambda(t)} \leq a_7\left\{%
		\begin{array}{ll}
			1+\ln(1+t), & \hbox{$\beta-r=2$;} \\
			2, & \hbox{$\beta -r> 2$;} \\
			(1+t)^{-\beta+r+2}, & \hbox{$1< \beta -r <2$ .} \\
		\end{array}%
		\right.
	\end{aligned}
\end{equation}We notice that condition
\eqref{dif} is satisfied if
\begin{equation}\label{new}
	(t+1)^{\beta} \Lambda(t) h_0(t) \vartheta(t) \leq a_8 \bigg(1- (\vartheta)^{\frac{1}{\beta}}   \bigg)^{\frac{\beta}{\beta
			+1}}.
\end{equation}
where $a_8 > 0$  depending on $a, \beta, c_1$ and $c_2$. Choosing
$\vartheta(t)$ as the following
\begin{equation}\label{gs3}\begin{aligned}
		&\vartheta(t) = \lambda\left\{%
		\begin{array}{ll}
			(1+t)^{-p},~~~~~ p=r+1 & \hbox{$\beta-r \geq 2$;} \\
			(1+t)^{-p},~~~~~ p=\beta-1, & \hbox{$1< \beta -r <2$ .} \\
		\end{array}%
		\right.
	\end{aligned}
\end{equation}
so that \eqref{xi} is valid. Moreover, using \eqref{gs0} and
\eqref{gs10}, we see that \eqref{new} is satisfied if $0 < \lambda
\leq 1$ is small enough, and then \eqref{dif} is satisfied. Hence
\eqref{decay1e} and \eqref{gs20} imply that, for any $t \in
\mathbb{R}_+$
\begin{equation}\label{gsw1}\begin{aligned}
		&E(t) \leq a_9\left\{%
		\begin{array}{ll}
			\bigg(1+\ln(1+t)\bigg)(1+t)^{-(\beta-r-1)}, & \hbox{$\beta-r=2$;} \\
			(1+t)^{-(\beta-r-1)}, & \hbox{$\beta-r>2$;} \\
			(1+t)^{-(\beta-r-1)}, & \hbox{ $1< \beta-r <2$ .} \\
		\end{array}%
		\right.
	\end{aligned}
\end{equation}
Thus,  the  estimate \eqref{gsw1} gives $\lim_{t\rightarrow +\infty}
E(t)=0$.\\  \textbf{Case 2:} if $m_0 \leq 1+\vert \vert \nabla
u_{0}\vert \vert^2 \leq m_1$. That is $r = 0$ in \eqref{r} as it was
assumed in many papers in the literature. So, it clear the estimate
\eqref{gsw1} gives $\lim_{t\rightarrow +\infty} E(t)=0$ when $r =
0$.

\section*{Acknowledgment} The authors thank King Fahd University of Petroleum and Minerals (KFUPM)  for its continuous supports. This work is supported by KFUPM under project \# SB201026.

\section*{Conflict of Interest}
The authors declare that there is no conflict of interest regarding the publication of this paper.

\end{document}